\documentclass[12pt]{amsart}

\usepackage{fullpage}
\usepackage{amsmath}
\usepackage{amsfonts}
\usepackage{amssymb}
\usepackage{amsthm}
\usepackage[all,knot,poly]{xy}
\usepackage{xypic}
\usepackage{hyperref}
\usepackage{graphicx}
\usepackage{MnSymbol}
\hypersetup{
    colorlinks=true,
    linkcolor=blue,
    filecolor=magenta,      
    urlcolor=cyan,
}
 
\input xy
\xyoption{all}

\theoremstyle{plain}

\newtheorem{theorem}{Theorem}[section]
\newtheorem{lemma}[theorem]{Lemma}
\newtheorem{proposition}[theorem]{Proposition}
\newtheorem{corollary}[theorem]{Corollary}
\newtheorem{remark}[theorem]{Remark}
\newtheorem{definition}[theorem]{Definition}

\usepackage{color}

\title{Sequential motion planning assisted by group actions}

\author{Emmett L. Balzer}
\author{Enrique Torres-Giese}

\address{Trinity Western University, Langley BC, V2Y 1Y1 , Canada.}
\email{emmett.balzer@mytwu.ca}
\email{enrique.torresgiese@twu.ca}

\begin{document}
\begin{abstract} 
We study higher analogues of effective and effectual topological complexity of spaces equipped with a group action. 
These are $G$-homotopy invariant and are motivated by the (higher) motion planning 
problem of $G$-spaces for which their group action is thought of as an external system assisting the motion planning. 
Related to this interpretation we define what we call orbital topological complexity, which is also a $G$-homotopy invariant that 
provides an upper bound for the topological complexity of the quotient space by the group action. We apply these concepts 
to actions of the group of order two on orientable surfaces and spheres.
\end{abstract}

\maketitle

%%%%%%%%%%%%%%%%%%%%%%%%%%%%%%%%%%%%%%%%%%%%%%%%%%%%%%%%%%%%%%%%%%%%%%%%%%%%%
%%%%%%%%%%%%%%%%%%%%%%%%%%%%%%%%%%%%%%%%%%%%%%%%%%%%%%%%%%%%%%%%%%%%%%%%%%%%%

\section{Introduction}

The study of the motion planning problem in topological spaces has been approached from a variety of perspectives 
that seek to take into account different facets of the problem such as the number of target points that must be visited, 
the symmetry of the motion planners, and the compatibility of the motion planners with inherent symmetries of the space where the motion occurs. 
In addition to these perspectives, we can also consider the motion planning problem of spaces equipped with a system of symmetries that can be 
incorporated into the motion planners. This latter situation is interpreted as having a topological space endowed with a group action. In terms of robotics, 
when considering the motion planning problem in a topological $G$-space $X$, we can think of the action of the group $G$ as an external system working on
$X$ that is meant to facilitate the motion planning. For example, if we consider $X = \mathbb{R}^2 - \{ (0,0) \}$ and the Lie group 
$O(2)$ acting on $X$ by rotations, then this action can be thought of as a crane revolving 
around the origin that lifts up an object, rotates, and then delivers the object in a different location within its $O(2)$-orbit. 
This interpretation has given rise to different flavors of the notion of topological complexity, and among these we have two that are
called effective and effectual topological complexity, defined in \cite{bk} and \cite{cggz}.  Recall that the (higher) $n$-th topological complexity (TC from now on) 
of a space $X$ is the sectional category of the multi-evaluation map 
\begin{equation}\label{evaluation_map_for_regular}
e_n :  M_n(X) \to X^n 
\end{equation}
where $M_n(X)$ stands for the space of multipaths in $X$ (see \cite{bgrt} and \cite{r}). 

In order to incorporate the action of the group $G$ on $X$ into the notion of TC, the effective TC 
replaces the space $M_2(X)$ by a space of ``broken" paths that can be glued when the action of $G$ is taken into account, and therefore produce 
a way of traveling from a source to a target point in $X$. On the other hand, the concept of effectual TC considers 
planners that connect the source point with a point in the $G$-orbit of the target point. 
The effective and effectual TC are denoted by $\mathsf{TC}^{G}_{\mathsf{effv}}(X)$ and $\mathsf{TC}^{G}_{\mathsf{effl}}(X)$ respectively and 

for free actions of finite groups 
they are related to the TC of the quotient space $X/G$ by the following inequalities (see~\cite{cggz}): 
\begin{equation}\label{inequality_of_three}
\mathsf{TC}^G_{\mathsf{effv}}(X) \leq \mathsf{TC}^G_{\mathsf{effl}}(X) \leq \mathsf{TC}_2(X/G).
\end{equation}

In this paper, we will provide natural generalizations of these concepts into the context of higher TC by
introducing an appropriate concept of broken multipaths for the case of effective TC, and more target $G$-orbits for the case
of effectual TC. The corresponding concepts are denoted by $\mathsf{TC}^{G}_{\mathsf{\mathsf{effv}},n}(X)$ and $\mathsf{TC}^{G}_{\mathsf{effl},n}(X)$. We also
introduce a third concept that we call orbital TC, denoted by $\mathsf{TC}^G_{\mathsf{orb},n}(X)$. This orbital TC
is motivated by the motion planning problem that considers multipaths in $X$ that reach points in the $G$-orbits of the desired target points. 
As in (\ref{inequality_of_three}), for free actions of finite groups these higher versions of TC are related as follows 
\begin{equation}\label{inequality_of_three_higher}
\mathsf{TC}^G_{\mathsf{effv},n}(X) \leq \mathsf{TC}^G_{\mathsf{effl},n}(X) \leq \mathsf{TC}_n(X/G) \leq \mathsf{TC}^G_{\mathsf{orb},n}(X).
\end{equation}

To further explore the concept of orbital TC and the relationship (\ref{inequality_of_three_higher}), we consider the antipodal and reflection actions on 
the sphere $S^m$ and three actions of $\mathbb{Z}_2$ on orientable surfaces: 
\begin{enumerate}
    \item reflection (for effective TC only),
    \item rotation (for odd genus surfaces), and
    \item antipodal.
\end{enumerate}

We summarize our calculations as follows:

\begin{theorem}
Let $S^m$ the $m$-dimensional sphere with $m\geq 1$, $\Sigma_g$ be a compact closed orientable surface of genus $g\geq 1$, 
and $\delta \in \{ 0, 1\}$. We have:
\[
\def\arraystretch{1.2}
\begin{array}{|l|c|c|c|}
\hline
 & \mathsf{TC}^{\mathbb{Z}_2}_{\mathsf{effv},n} & \mathsf{TC}^{\mathbb{Z}_2}_{\mathsf{effl},n} & \mathsf{TC}^{\mathbb{Z}_2}_{\mathsf{orb},n}\\
\hline
\mathsf{Reflection} & & &   \\
S^m & n & & \\
\Sigma_1 & 2n-1 & &   \\
\Sigma_g \ \ (g\geq 2) & 2n +\delta  & &   \\
\hline
\mathsf{Rotation} & & &   \\
\Sigma_1 & 2n-1 & 2n-1 & 2n+\delta  \\
\Sigma_{2l+1} \ \ (l\geq 1) & 2n+1 & 2n+1 & 2n+1  \\
\hline
\mathsf{Antipodal} & & &   \\
S^m & n & (n-1)m +1 + \delta  & nm+1\\
\Sigma_1 & 2n - 1 & 2n & 2n + 1  \\
\Sigma_g \ \ (g\geq 2)  & 2n + \delta & 2n + \delta & 2n + 1  \\
\hline
\end{array}
\]
\end{theorem}

The previous result extends the calculations when $n=2$ available in \cite{bk} and \cite{cggz}. 
The value of $\delta$ can be settled in some cases, for instance, we know that $\mathsf{TC}_n(\mathbb{R}P^m)=m(n - 1)+1$ when $m\in \{1,2,3\}$,
and according to (\ref{inequality_of_three_higher}) this implies that $\mathsf{TC}^{\mathbb{Z}_2}_{\mathsf{effl},n}(S^m) = \mathsf{TC}_n(\mathbb{R}P^m)$ 
when $m\in \{1,2,3\}$.  In general, settling the value(s) of $\delta$ would require devising an appropriate motion planner forcing $\delta = 0$, or the use of 
obstruction theory (that will appear elsewhere). Returning to the antipodal action on the sphere, the inequalities in (\ref{inequality_of_three_higher}) become: 
\[ 
\mathsf{TC}^{\mathbb{Z}_2}_{\mathsf{effv},n}(S^m) \leq \mathsf{TC}^{\mathbb{Z}_2}_{\mathsf{effl},n}(S^m) \leq 
\mathsf{TC}_n(\mathbb{R}P^m) \leq \mathsf{TC}^{\mathbb{Z}_2}_{\mathsf{orb},n}(S^m).
\]
For these spaces, the first of the inequalities is always strict. However, $\mathsf{TC}_n(\mathbb{R}P^m)$ could be equal $\mathsf{TC}^{\mathbb{Z}_2}_{\mathsf{effl},n}(S^m)$ 
as we already mentioned; and $\mathsf{TC}_n(\mathbb{R}P^m)$ could be also equal to $\mathsf{TC}^{\mathbb{Z}_2}_{\mathsf{orb},n}(S^m)$, for
instance when $m= 2^e$ and $n\geq 3$ (see ~\cite{cgggl}). This latter implies that the last two inequalities in (\ref{inequality_of_three_higher}) are sharp. 
There are instances when (\ref{inequality_of_three_higher}) yields strict inequalities:
\[ 
\mathsf{TC}^{\mathbb{Z}_2}_{\mathsf{effv},2}(S^4)  = 2,\hspace{0.5cm} \mathsf{TC}^{\mathbb{Z}_2}_{\mathsf{effl},2}(S^4)=5+\delta,\hspace{0.5cm} 
\mathsf{TC}_2(\mathbb{R}P^4) = 8,\hspace{0.5cm} \mathsf{TC}^{\mathbb{Z}_2}_{\mathsf{orb},2}(S^4) = 9.
\]
The third value above follows from the celebrated equality $\mathsf{TC}_2(\mathbb{R}P^m) = \mathsf{Imm}(\mathbb{R}P^m)+1$ when $m\neq 1,3,7$ (here 
$\mathsf{Imm}(\mathbb{R}P^m)$ is the immersion dimension of $\mathbb{R}P^m$, which is roughly $2m$ minus twice the number of ones in the dyadic expansion of $m$). 
As we have seen, the effectual and the orbital TC serve as approximations to the TC of the quotient space, and it is plausible that both of these approximations may yield the
same value as the TC of the quotient space, but we do not know an example of this yet.

The organization of this paper is as follows. In Sections \ref{higher_effv} and \ref{Higher_effl_orb} we define the higher effective, effectual and orbital TC. Then in Section \ref{properties}
we present general properties of these invariants, and in Sections \ref{spheres} and \ref{surfaces} we study these invariants for spheres and surfaces. We also include an
appendix with a motion planner for euclidean spaces minus a number of points. 

The spaces that we will consider in this paper will all be Hausdorff, path-connected, and locally path-connected; and the groups
acting on these spaces will all be finite and assumed to act freely. All the results in this paper concerning TC, Lusternik-Schnirelmann category, and sectional category are non-reduced.

%%%%%%%%%%%%%%%%%%%%%%%%%%%%%%%%%%%%%%%%%%%%%%%%%%%%%%%%%%%%%%%%%%%%%%%%%%%%%
%%%%%%%%%%%%%%%%%%%%%%%%%%%%%%%%%%%%%%%%%%%%%%%%%%%%%%%%%%%%%%%%%%%%%%%%%%%%%

\section{Higher Effective Topological Complexity}\label{higher_effv}

Given a $G$-space $X$, there are two closely related approaches that define ``effective'' TC of $X$: 
the construction of \cite{bk} that uses the space of broken paths $P^{G,n}(X)$ that can be patched with elements in $G$; 
and the construction of \cite{cggz} that uses the space $P^G_n(X)$ consisting of broken paths that can be patched with
elements in $G$ along with these patching data. We will make these approaches more precise in the following two 
subsections where we introduce their higher analogues.

\subsection{Higher Effective TC with patching data}
To propose a higher version of effective topological complexity we must consider a suitable replacement of the space of multipaths. 
We do so by means of the following definition.

\begin{definition} Let $M^G_n(X)$ be the subspace of $PX\times G \times PX\times\cdots\times PX\times G\times PX$ consisting of tuples $(\alpha_1,g_1,\alpha_2,g_2,\ldots,\alpha_{n-1},g_{n-1},\alpha_n)$ such that $\alpha_i(0)g_i = \alpha_{i+1}(0)$ for $1\leq i \leq n-1$.
\end{definition}

An element in the space $M^G_n(X)$ consists of an $n$-tuple of paths $\alpha_1,\ldots,\alpha_n$ in $X$ emanating from the $G$-orbit of a point in $X$ along with patching data provided by the group elements $g_1,\ldots,g_{n-1}$. Note that we have an evaluation map
\begin{equation}\label{evaluation_map_for_effective}
\epsilon_n : M^G_n(X) \to X^n.
\end{equation}
given by $(\alpha_1,g_1,\ldots,g_{n-1},\alpha_n) \mapsto (\alpha_1(1),\ldots,\alpha_n(1))$. 
This a natural extension of the space $M_n(X)$ since when $G$ is the trivial group the space $M^G_n (X)$ is precisely the space of 
multipaths $M_n(X)$.

\begin{proposition} The evaluation function $\epsilon_n : M^G_n X \to X^n$ is a fibration. Moreover, its sectional category is the 
same as that of the saturated diagonal map $\Delta_G^n : X \times G^{n-1} \to X^n$ given by 
$(x,g_1,\ldots,g_{n-1}) \mapsto (x, xg_1,xg_1g_2\ldots,xg_1\cdots g_{n-1})$. 
\end{proposition}
\begin{proof}
Consider the following diagram.
\[
\xymatrix{
M^G_n(X) \ar[rr]^\mu & & Q \ar[d]  \ar[rrd]^{ev_1} \ar[rr] & & P(X^n) \ar[d]^{ev_0} \\
 & & X \times G^{n-1} \ar[rr]_{\Delta_G^n} & & X^n .
}
\]
The square is the classical construction to replace the saturated diagonal by a fibration, and the map $\mu$ is a homeomorphism, 
which is given by 
\[(\alpha_1,g_1,\ldots,g_{k-1},\alpha_n) \mapsto (\alpha_1(0), g_1,\ldots,g_{n-1},\alpha_1,\ldots,\alpha_n)\]
with inverse $(x,g_1,\ldots,g_{n-1},\alpha_1,\ldots,\alpha_n) \mapsto (\alpha_1,g_1,\alpha_2,\ldots,g_{n-1},\alpha_n)$.
The evaluation map $\epsilon_n$ is the composition $\mu$ followed by $ev_1$. 
The result follows.
\end{proof}

Recall that the sectional category $\mathsf{secat}(p)$ of a fibration $p:E\to B$ is the smallest number of open sets $U_1,\ldots,U_n$ that cover $B$ and
such that there is a section $s_i$ of $p$ on each $U_i$.

\begin{definition}
The $n$-th effective topological complexity of $X$, denoted by $\mathsf{TC}^G_{\mathsf{effv},n}(X)$, is the sectional category of
the evaluation map $\epsilon_n : M^G_n(X) \to X^n$.
\end{definition}

In terms of the motion planning problem, a local section of $\epsilon_n$ over $U\subseteq X^n$ provides an algorithm to visit each 
component of a tuple of points in $U$ using paths in $X$ that emanate from the $G$-orbit of a point in $X$. 
The patching information provided by the group elements can be encoded in different ways. For instance,
we could consider the space $\mathcal{M}^G_n(X)$ consisting of tuples
$(\alpha_1,g_1,\ldots,g_{n-1},\alpha_n)$ such that $\alpha_1(0)g_i = \alpha_{i+1}(0)$ for $1\leq i\leq n-1$. The space $\mathcal{M}^G_k$ 
is equipped with an evaluation map $\epsilon'_n:\mathcal{M}^G_n(X) \to X^n$. It is easy to see that there is a homeomorphism 
$\phi:\mathcal{M}^G_k\to M^G_k$ that makes the following diagram commutative
\[
\xymatrix{
\mathcal{M}^G_n(X) \ar[rd]_{\epsilon'_n} \ar[rr]^\phi & & M^G_n(X) \ar[ld]^{\epsilon_n} \\
 & X^n &
}
\]
Therefore, both maps $\epsilon_n$ and $\epsilon'_n$ have the same sectional category. Moreover, it is easy to see that there is a homeomorphism 
between $M^G_n(X)$ and $M_n(X)\times G^{n-1}$ that could be used to carry the patching data in different ways.  

Under the homeomorphism $M^G_n(X) \cong M_n(X)\times G^{n-1}$, any restriction of the saturated diagonal to a copy of $M_nG$ different from
$M_n(X)\times \{ (1,\ldots,1)\}$, will be called a twisted diagonal.

\begin{remark} \normalfont
The definition of $\mathsf{TC}^G_{\mathsf{effv},n}(X)$ is compatible with that of \cite{cggz} when $n=2$. There are only two minor differences. 
The first is that the effective topological complexity of \cite{cggz} is in terms of broken paths consisting of sequences of paths, but in \cite[Proposition 2.5]{cggz} 
it is proved that it suffices to take broken paths with only two components. 
The second occurs in \cite[Proposition 2,2]{cggz} where they use a ``twisted" evaluation map in order to show that the map that they use to 
define effective topological complexity is indeed a fibration. In our case this is not necessary because we are working with broken multipaths
emanating from a $G$-orbit and we are evaluating them at 1, while \cite{cggz} considers broken paths with two components, the first of which has 
its end point in the $G$-orbit of the start point of the second component -- this is precisely what demands the use of the twisted evaluation map.
\end{remark}

%%%%%%%%%%%%%%%%%%%%%%%%%%%%%%%%%%%%%%%%%%
\subsection{Effective TC without patching data}

We can also define a notion of effective topological complexity without the patching data, as in~\cite{bk} when $n=2$.
The space that we will consider here is the space $M^{G,n}(X)$ consisting of broken multipaths emanating from the same 
$G$-orbit (without the patching group elements). More precisely, 
$M^{G,n}(X) = \{ (\alpha_1,\ldots,\alpha_n)\in (PX)^n: \alpha_i(0)G = \alpha_j(0)G \}$.
Note that we have an evaluation map $\tilde{\epsilon}_n: M^{G,n}(X) \to X^n$ given by $(\alpha_1,\ldots,\alpha_n)\mapsto (\alpha_1(1),\ldots,\alpha_n(1))$. 
The map $F:M^G_n(X) \to M^{G,n}(X)$ that forgets the patching data makes the following diagram commute.
\[
\xymatrix{
M^{G,n}(X) \ar[rr]^F \ar[rd]_{\tilde{\epsilon}_n} & & M^G_n(X) \ar[ld]^{\epsilon_n} \\
 & X^ n &
}
\]
The sectional category of $\tilde{\epsilon}_n$ would define a notion of higher effective topological complexity without patching data (to
see that $\tilde{\epsilon}_n$ is a fibration one can apply the argument of \cite[Proposition 3.7]{lm}). 
Note that this latter is less than or equal to $\mathsf{TC}^G_{\mathsf{effv},n}(X)$. We are not going to explicitly define this version of topological complexity 
to avoid overloading the notation. However, we will note that this notion agrees with effective topological complexity with patching data when the action of $G$ on $X$ is principal.

\begin{lemma}
Let $\pi_i:G^{n-1} \to G$ be the projection onto the $i$-th factor.
If the action of $G$ on $X$ is principal, then the function
$\tau_n:\Delta^G_n(X) \to G^{n-1}$ such that $x_i \pi_i\circ\tau_n(x_1,\ldots,x_n) = x_{i+1}$ for $1\leq i \leq n-1$, is continuous.
\end{lemma}
\begin{proof}
Recall that $\Delta^G_n(X) = \{(x,xg_1,\ldots,xg_1\cdots g_{n-1}) : x\in X, g_i\in G \} \subseteq X^n$.
By hypothesis, we know that the action of $G$ on $X$ is free and that $\tau_2$ is continuous. The continuity of 
$\tau_n$ follows from the following commutative diagram.
\[
\xymatrix{
\Delta^G_n(X) \ar[d]^{\tau_n} \ar[rr]^{\pi_{i,i+1}} & & \Delta^G_2(X) \ar[d]^{\tau_2} \\
X^n \ar[rr]^{\pi_i} & & X
}
\]
\end{proof}

\begin{lemma}
If the action of $G$ on $X$ is principal, then the map $F:M^G_n(X) \to M^{G,n}(X)$ is a homeomorphism.
\end{lemma}
\begin{proof}
For simplicity of notation, we will write $\bar{\alpha}(0)$ to stand for $(\alpha_1(0),\ldots,\alpha_n(0))$, where $\alpha_j \in P(X)$. 
The inverse of $F$ is given by 
$(\alpha_1,\ldots,\alpha_n)\mapsto (\alpha_1,\pi_1\circ\tau_n(\bar{\alpha}(0)),\ldots,\pi_{n-1}\circ\tau_n(\bar{\alpha}(0)),\alpha_n)$. 
\end{proof}

\begin{corollary}\label{effective_values_are_the_same}
If the action of $G$ on $X$ is principal, then effective topological complexity with or without patching data yield the same value.
\end{corollary}

%%%%%%%%%%%%%%%%%%%%%%%%%%%%%%%%%%%%%%%%%%%%%%%%%%%%%%%%%%%%%%%%%%%%%%%%%%%%%
%%%%%%%%%%%%%%%%%%%%%%%%%%%%%%%%%%%%%%%%%%%%%%%%%%%%%%%%%%%%%%%%%%%%%%%%%%%%%

\section{Higher Effectual and Orbital TC}\label{Higher_effl_orb}
 
 The concept of a higher effectual topological complexity of a $G$-space $X$ in the context of the motion planning problem is described as follows: suppose that we are given a tuple $(x_1,[x_2],\ldots,[x_n])$ of target points $x_1\in X$ and $[x_i]\in X/G$ that we want a robot to visit them, 
 and that this task is fulfilled if we can provide a multipath in $X$ whose first component reaches $x_1$ but the remaining components only reach points in $X$ that belong to 
 the $G$-orbits $[x_2],\ldots,[x_n]$. This approach assumes that the motion planning is assisted by the group action, which can be thought of 
 as a mechanical  system facilitating the motion in $X$ as we mentioned in the introduction. This specific motion planning is then related to the
 sectional category of  the map 
 \begin{equation}\label{evaluation_map_for_effectual}
\varepsilon_n : M_n(X) \to X\times (X/G)^{n-1}
 \end{equation}
 defined as the composition of the evaluation map $e_n:M_n(X) \to X^n$ followed by 
 $1\times \pi^{n-1}: X^n \to X\times (X/G)^{n-1}$. Note that if $\pi: X\to X/G$ is a
 fibration, then so is $\varepsilon_n$. Whenever we speak of effectual topological complexity we will consider only free actions of finite 
 groups.
 
\begin{definition}
The $n$-th effectual topological complexity of $X$, denoted by $\mathsf{TC}^G_{\mathsf{effl},n}(X)$, is the sectional category of $\varepsilon_n : M_n(X) \to X\times (X/G)^{n-1}$.
\end{definition}

Related to the effectual motion planning problem, we can also consider the following scenario: 
suppose that we are given a tuple $(x_1,\ldots,x_n)$ of target points in $X$ and that we want to provide a multipath to a robot to visit them 
or at least a member from their $G$-orbit in $X$. This corresponds to studying the sectional category of the map
\begin{equation}\label{evaluation_map_for_orbital}
\mathsf{e}_n: M_n(X) \to (X/G)^n
\end{equation}
defined as the composition $\pi^n\circ e_n: M_n(X) \to X^n \to (X/G)^n$. We will consider only the case when $G$ is a finite group acting freely on $X$.

\begin{definition}
The $n$-th orbital topological complexity of $X$, denoted by $\mathsf{TC}^G_{\mathsf{orb},n}(X)$, is the sectional category of the composite
$\mathsf{e}_n: M_n(X) \to X^n \to (X/G)^n$.
\end{definition}

The following result is a generalization of \cite[Proposition 3.3]{cggz}. 
  
\begin{proposition}\label{inequality}
If $G$ acts freely on $X$, then 
\[ \mathsf{TC}^G_{\mathsf{effv},n}(X) \leq \mathsf{TC}^G_{\mathsf{effl},n}(X) \leq \mathsf{TC}_n(X/G) \leq \mathsf{TC}^G_{\mathsf{orb},n}(X). 
\]
\end{proposition}
\begin{proof}
To prove the the first two inequalities consider the following commutative diagram.
\[
\xymatrix{
M^G_n (X) \ar[rr]^c \ar[d]^{e_n} & & M_n (X) \ar[rr]\ar[d]^{1\times \pi^{n-1}} & & M_n(X/G) \ar[d]^{e_n} \\
X^n \ar[rr]_{1\times \pi^{n-1}} & & X \times (X/G)^{n-1} \ar[rr]_{\pi \times 1^{n-1}} & & (X/G)^{n} 
}\]
The map $c$ is given by 
$(\alpha_1,g_1,\alpha_2,\ldots,g_{n-1},\alpha_n)\mapsto (\alpha_1,\alpha_2 g_1^{-1},\ldots,\alpha_n g_n^{-1} \cdots g_1^{-1})$. 
The argument of \cite[Proposition 3.3]{cggz} can be used here to show that both squares in this latter diagram are pullbacks. We
leave the details to the reader. The last inequality can be obtained from the following commutative diagram
\[
\xymatrix{
M^n (X) \ar[rr] \ar[dr]_{\mathsf{s}_n} & & M_n (X/G) \ar[ld]^{e_n} \\
 & (X/G)^n & 
}\]
where the top map is induced by the quotient map $\pi:X\to X/G$.
\end{proof}

%%%%%%%%%%%%%%%%%%%%%%%%%%%%%%%%%%%%%%%%%%%%%%%%%%%%%%%%%%%%%%%%%%%%%%%%%%%%%
%%%%%%%%%%%%%%%%%%%%%%%%%%%%%%%%%%%%%%%%%%%%%%%%%%%%%%%%%%%%%%%%%%%%%%%%%%%%%

\section{Properties}\label{properties}

In this section we state basic properties of higher effective, effectual, and orbital TC. 
Several properties for higher effectual TC can be derived from ~\cite{pp} by applying the 
concept of topological complexity of a map to $\pi: X\to X/G$. This latter concept also has been developed into a higher 
version in ~\cite{ik}. All the properties in the result below are standard and can be proved following the definitions and the 
arguments used to prove the analogous properties in the case when $n=2$ or the arguments in the case of higher regular TC. 
We leave the details to the reader.

\begin{theorem}
Suppose that $X$ is a $G$-space. In addition, assume that $G$ acts freely on $X$ for statements related to effectual and orbital TC.
\begin{enumerate}
\item $\mathsf{TC}^G_{\mathsf{effv},n}(X) \leq \mathsf{TC}^G_{\mathsf{effv},n+1}(X)$, and $\mathsf{TC}^G_{\mathsf{orb},n}(X) \leq \mathsf{TC}^G_{\mathsf{orb},n+1}(X)$.
\item If $f: X \to Y$ is a $G$ map with a homotopy right inverse $g:Y\to X$, then $\mathsf{TC}^G_{\mathsf{effv},n}(Y) \leq \mathsf{TC}^G_{\mathsf{effv},n}(X)$, 
and $\mathsf{TC}^G_{\mathsf{effl},n}(Y) \leq \mathsf{TC}^G_{\mathsf{effl},n}(X)$.
Therefore, both $\mathsf{TC}^G_{\mathsf{effv},n}$ and $\mathsf{TC}^G_{\mathsf{effl},n}$ are a $G$-homotopy invariant.
\item If $f: X \to Y$ and $g:Y\to X$ are $G$-maps so that $g$ is a homotopy right inverse of $f$, then $\mathsf{TC}^G_{\mathsf{orb},n}(Y) \leq \mathsf{TC}^G_{\mathsf{orb},n}(X)$. 
Therefore, $\mathsf{TC}^G_{\mathsf{orb},n}$ is a $G$-homotopy invariant.
\item If $H \leq G$, then $\mathsf{TC}^G_{\mathsf{effv},n}(X) \leq \mathsf{TC}^H_{\mathsf{effv},n}(X) \leq \mathsf{TC}_n(X)$.
\item $\mathsf{TC}^G_{\mathsf{effv},n}(X) \leq \mathsf{cat}(X^n) \leq \frac{n \mathrm{hdim}(X)}{\mathrm{conn}(X)+1} + 1$.
\item $\mathsf{TC}_n(X/G) \leq \mathsf{TC}^G_{\mathsf{orb},n}(X)\leq \mathsf{cat}((X/G)^n)$.
\item $\mathsf{cat}((X/G)^{n-1}) \leq \mathsf{TC}^G_{\mathsf{effl},n}(X)$.
\end{enumerate}
\end{theorem}

\begin{remark}\normalfont
As we mentioned in the introduction, there are different versions of TC that make use of the action of
a $G$-space: those that either are compatible with the $G$-action, or those that incorporate the group action as part of the motion planning. 
The reader can consult ~\cite{bk} to see a brief summary of the different versions that have been developed. Among these versions 
we have the so-called invariant topological complexity, which is related to effective topological complexity but different
as pointed in ~\cite{bk}. The interested reader could also go to \cite{bs} to see a higher version of invariant topological
complexity, which is also different from the higher effective topological complexity developed in this paper. 
The relationship between these two concepts has been already pointed out in ~\cite{bk} (for the case $n=2$).
\end{remark}

According to \cite[Proposition 2.2]{r}, since the maps $e_n$, $\epsilon_m$, $\varepsilon_n$, $\mathsf{e}_n$ are fibrations (defined in (\ref{evaluation_map_for_regular}), 
(\ref{evaluation_map_for_effective}), (\ref{evaluation_map_for_effectual}) and (\ref{evaluation_map_for_orbital}) respectively), 
the domains of their local planners (i.~e. local sections) can be replaced by ENRs when the corresponding base space of these maps is a polyhedron. Moreover,
when the base space of these maps is an ENR, the openness condition of the domains of the local planners can be replaced by requiring the domains to 
form a partition of ENRs (see for instance \cite[Theorem 4.9]{farber}).  

The following result that comprises Theorems 4 and 5’  in \cite{sv} will be applied to the maps, 
(\ref{evaluation_map_for_effective}), (\ref{evaluation_map_for_effectual}), and (\ref{evaluation_map_for_orbital}) to obtain lower bounds for the
corresponding TC that they afford. A nontrivial element in the kernel of the homomorphism that they induce in cohomology will be called 
an effective, effectual, and orbital zero-divisor respectively.

\begin{proposition}
Let $p:E\to B$ be a fibration, and let $u_1\ldots,u_m$ be classes in $H^*(B;R)$ in the kernel of the homomorphims induced in cohomology 
by $p$. If $u_1\cdots u_m \neq 0$, then $\mathsf{secat}(p) \geq m+1$.  
\end{proposition}

%%%%%%%%%%%%%%%%%%%%%%%%%%%%%%%%%%%%%%%%%%%%%%%%%%%%%%%%%%%%%%%%%%%%%%%%%%%%%
%%%%%%%%%%%%%%%%%%%%%%%%%%%%%%%%%%%%%%%%%%%%%%%%%%%%%%%%%%%%%%%%%%%%%%%%%%%%%

\section{TC and Symmetries of Spheres}\label{spheres}

In this section we will consider the antipodal and reflection action on the sphere $S^m$, with $m\geq 1$.
The calculation of $\mathsf{TC}^{\mathbb{Z}_2}_{\mathsf{effv},2}(S^m)$ was obtained in \cite{bk}, and the value
of $\mathsf{TC}^{\mathbb{Z}_2}_{\mathsf{effl},2}(S^m)$ was considered in \cite{cggz}. Here we extend these results.
Given $v\in S^m$, we let $c_v$ be the constant path at $v$, and if $u\neq -v$ we denote by $[u,v]$ the shortest path in $S^m$ 
from $u$ to $v$.

\begin{proposition}
If $\mathbb{Z}_2=\langle \rho \rangle$ acts on $S^m$ by $\rho(x_1,\cdots,x_{m+1})=(x_1,\cdots,x_m,-x_{m+1})$, then
 \[
 \mathsf{TC}^{\mathbb{Z}_2}_{\mathsf{effv},n}(S^m) = n.
 \]
\end{proposition} 
\begin{proof}
Let $w$ be the generator of degree $2$ of the mod-2 cohomology of $S^m$, and note that the action of $\mathbb{Z}_2$ on the mod-2 
cohomology of $S^m$ is trivial. When $n=2$, the class $w\otimes w$ is an effective zero divisor. When $n>2$ consider the classes 
$W_j = w\otimes 1 \otimes \cdots \otimes 1 + 1\otimes \cdots \otimes w\otimes 1\otimes \cdots\otimes 1$, where $w$ is in the $j$-th 
factor of the second summand. The classes $U_2,\cdots, U_n$ are nontrivial effective zero divisors and have nontrivial product. 
Hence $n\leq \mathsf{TC}^{\mathbb{Z}_2}_{\mathsf{effv},n}(S^m)$, and since we know that $\mathsf{TC}^{\mathbb{Z}_2}_{\mathsf{effv},n}(S^m) \leq \mathsf{TC}_n(S^m)$ it suffices 
to provide an effective motion planner in the case when $m$ is even. 

Let us assume that $m$ is even and pick a tangent vector field of $S^{m-1}$ that we will propagate on $S^m$ in such a way that defines a 
tangent field $\nu$ which is orthogonal to the $x_{m+1}$-axis and vanishes at the north and south poles $(0,\ldots,0,\pm 1)$. 

Let $V_j$ be the subspace of $(S^m)^n$ consisting of tuples $(v_1,\dots,v_n)$ such that there are exactly $j$ components in
$(v_2,\ldots,v_n)$ that are equal to  $-\rho(v_1)$. Thus $V_0,\ldots,V_{n-1}$ form a partition of $(S^m)^n$. Given $v\in S^m$
we define $\beta(v)$ to be the path from $v$ to $-\rho(v)$ determined by $\nu(v)$ and contained in the hyperplane $x_{m+1}=p_{m+1}(v)$, where 
$p_{m+1}$ is the projection onto the $m+1$ factor.
Given $v,w \in S^m$ we let $\alpha(v,w) \in G \times P(S^m)$ be given by
\[
 \alpha(v,w) = \left\{
\begin{array}{lr}
(\rho,[\rho(v),w]) & \text{if } w \neq -\rho(v) \\
(1,\beta(v)) & \text{if } w = -\rho(v)
\end{array}
 \right.  .
 \]
Thus we define $s(v_1,\ldots,v_n) = (c_{v_1},\alpha(v_1,v_2),\ldots,\alpha(v_1,v_n))$. The restriction of $s$ to each of the subspaces $V_j$ 
defines a motion planner that yields multipaths emanating from the orbit of $v_1$.
This motion planner provides the desired upper bound $ \mathsf{TC}^{\mathbb{Z}_2}_{\mathsf{effv},n}(S^m) \leq n$ when $n$ is even.
\end{proof}

 \begin{proposition}
 If $\mathbb{Z}_2=\langle \sigma \rangle$ acts on $S^m$ by $\sigma(x) = -x$, then
 \[
 \mathsf{TC}^{\mathbb{Z}_2}_{\mathsf{effv},n}(S^m) =n.
 \]
 \end{proposition}
 \begin{proof}
 Let $U_j$ be the subspace of $(S^m)^n$ consisting of tuples $(x_1,\ldots,x_n)$ with exactly $j$ components that are different from $x_1$. Note that each $U_j$ is an ENR and that $(S^m)^n = U_0 \sqcup \ldots \sqcup U_{n-1}$. For simplicity of notation, since the action is principal, we will describe effective motion planners without the patching data. 
 
 Let $\alpha(x,y)$ be the path given by
 \[
 \alpha(x,y) = \left\{
\begin{array}{lr}
[-x,y] & \text{if } y \neq x \\
c_x & \text{if } y = x
\end{array}
 \right.  .
 \]
Now we let $s:(S^m)^n \to M^{\mathbb{Z}_2,n}(S^m)$ given by
\[
s(x_1,\ldots,x_n) = (\alpha(x_1,x_1), \alpha(x_1,x_2),\ldots,\alpha(x_1,x_n)).
\]
Thus, the restriction of $s$ to each $U_j$ determines an effective planner. 
Hence $\mathsf{TC}^{\mathbb{Z}_2}_{\text{effv},n}(S^m) \leq n$. To complete the proof, we will consider zero divisors for the map 
$\epsilon_n$ (these are called effective zero divisors in~\cite{cg}). Recall that $\epsilon_n$ is the fibrational substitute of the 
saturated diagonal map $\Delta^{\mathbb{Z}_2}_n: S^m\times \mathbb{Z}_2^{n-1} \to (S^m)^n$. Let $\sigma_i$ be the $(n-1)$-tuple in
$\mathbb{Z}_2^{n-1}$ with $i^{th}$ entry equal to $\sigma$ and all the remaining entries equal to 1, and let 
$j_{i_1,\ldots,i_r}$ be the natural inclusion map $S^m\cong S^m\times \{ \sigma_{i_1}\cdots\sigma_{i_r}\} \to S^m \times \mathbb{Z}_2^{n-1}$.
Note that in mod-2 cohomology, each composite $\Delta^{\mathbb{Z}_2}_n \circ j_{i_1,\ldots,i_r}$ induces the same homomorphism as 
the diagonal map $\Delta_n: S^m \to (S^m)^n$. The result follows since an element in the kernel of the diagonal will be in the of all the 
inclusions $j_{i_1,\ldots,i_r}$, and it's well-known that in mod-2 cohomology the zero divisors cup length of $\Delta_n$ 
equals $n-1$. Therefore $\mathsf{TC}^{\mathbb{Z}_2}_{\mathsf{effv},n}(S^m) \geq n$ as wanted.
\end{proof}

\begin{proposition}
If $\mathbb{Z}_2$ acts antipodally on $S^m$, then
\[
\mathsf{TC}^{\mathbb{Z}_2}_{\mathsf{effl},n}(S^m) \in \{ (n-1)m + 1,(n-1)m+2\}.
\]
\end{proposition}
\begin{proof}
We have $ \mathsf{cat}( (\mathbb{R}P^m)^{n-1}) \leq \mathsf{TC}^{\mathbb{Z}_2}_{\mathsf{effl},n}(S^m) \leq \mathsf{cat}(S^m\times (\mathbb{R}P^m)^{n-1})$. 
The result follows.
\end{proof}

\begin{proposition}
If $\mathbb{Z}_2$ acts antipodally on $S^m$, then
\[
\mathsf{TC}^{\mathbb{Z}_2}_{\mathsf{orb},n}(S^m) = nm+1.
\]
\end{proposition}
\begin{proof}
Note that the kernel of $(\pi^n)^*:H^*( (X/G)^n ) \to H^*(X^n)$ is contained in the kernel of $(\pi^n\circ e_n)^*: H^*( (X/G)^n) \to H^*(X)$.
So, if $\alpha$ is the generator of $H^1(\mathbb{R}P^m;\mathbb{Z}_2)$ and $p_i: (\mathbb{R}P^m)^n \to \mathbb{R}P^m$ is the projection onto
the $i$-th factor, the the classes $p_i(\alpha)$ are in the kernel of $(\pi^n)^*$ and $(p_1(\alpha))^m \cdots p_n(\alpha)^m \neq 0$. Moreover, since
$\mathsf{TC}^G_{\mathsf{orb},n}(S^m)\leq \mathsf{cat}((\mathbb{R}P^m)^n)$, it follows that $\mathsf{TC}^{\mathbb{Z}_2}_{\mathsf{orb},n}(S^m) \leq nm+1$.
This completes the proof.
\end{proof}

Note that this last result shows that the topological complexity of the quotient space $X/G$ is not equal to the orbital topological complexity of $X$.
For instance, it is know that $TC_2(\mathbb{R}P^3)=\mathsf{cat}(\mathbb{R}P^3)=4$, while $\mathsf{TC}^{\mathbb{Z}_2}_{orb,2}(S^3) = 7$. 

%%%%%%%%%%%%%%%%%%%%%%%%%%%%%%%%%%%%%%%%%%%%%%%%%%%%%%%%%%%%%%%%%%%%%%%%%%%%%
%%%%%%%%%%%%%%%%%%%%%%%%%%%%%%%%%%%%%%%%%%%%%%%%%%%%%%%%%%%%%%%%%%%%%%%%%%%%%

\section{TC and Symmetries of Surfaces}\label{surfaces}

In this section we will consider three different actions of $\mathbb{Z}_2$ on an orientable surface $\Sigma_g$ of genus 
$g\geq 1$ embedded in $\mathbb{R}^3$ according to Figure \ref{fig:Prezel}. 
%For these three actions we will consider $\Sigma_g$ 
%as embedded in $\mathbb{R}^3$ with its center of mass at the origin and %symmetric about the $XY$-plane and the $XZ$-plane.
The three actions are:
\begin{enumerate}
    \item The reflection action: this is given by reflection about the $xy$-plane. This action is not free, and has a 
    quotient space homotopy equivalent to a bouquet of $g$ circles.
    \item The rotation action: this is given by rotating $\Sigma_g$ about the $z$-axis by 180 degrees. We will consider this action only when $g$ is odd, which is the case when this action is free and $\Sigma_{2l+1}$ has as quotient $\Sigma_{l+1}$.
    \item The antipodal action: this is given by $(x,y,z)\mapsto (-x,-y,-z)$. This is a free action with
    quotient space a non-orientable surface $N_{g+1}$ of genus $g+1$. 
\end{enumerate}

When regarding $\Sigma_g$ with one of these actions we will write $\Sigma_g^{\mathsf{ant}}$, $\Sigma_g^{\mathsf{rot}}$, and $\Sigma_g^\mathsf{ref}$ accordingly.
Recall that the (integral) cohomology of $\Sigma_n$ has generators in degree one $a_1,\ldots,a_n$ and $b_1,\ldots,b_n$, and one generator $c$ in degree two such 
that $a_i^2=b_i^2=0$ and $a_i b_i=\delta_{ij} c$. 

%\begin{figure}
%    \centering
%    \includegraphics[height=4cm]{Reports/Overview/Sausage Middle 3-4.png}
%    \caption{Embedding in $\mathbb{R}^3$ when $g$ is odd or even respectively.}
%    \label{fig:Sausage}
%\end{figure}

\begin{figure}
    \centering
    \includegraphics[height=5.8cm]{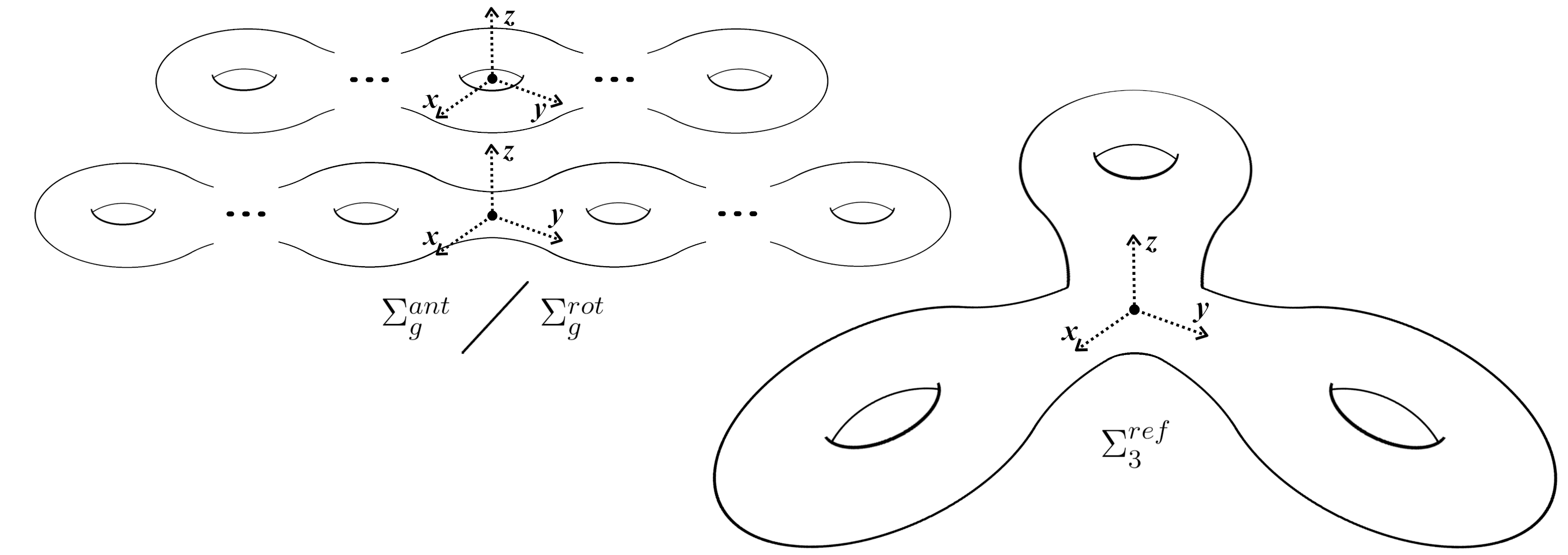}
    \caption{Embedding $\Sigma_g^{\mathsf{ant}}$, $\Sigma_g^{\mathsf{rot}}$, and $\Sigma_3^\mathsf{ref}$ in $\mathbb{R}^3$.}
    \label{fig:Prezel}
\end{figure}

%%%%%%%%%%%%%%%%%%%%%%%%%%%%%%%%%%%%%%%%%%%%%%%%%%%%%%%%%%%%%%%%%%%%%%%%%%%%%
%%%%%%%%%%%%%%%%%%%%%%%%%%%%%%%%%%%%%%%%%%%%%%%%%%%%%%%%%%%%%%%%%%%%%%%%%%%%%

\subsection{The reflection action}

Since the reflection action on $\Sigma_g$ is not free, we will consider only the effective topological complexity of $\Sigma_g^\mathsf{ref}$.

\begin{proposition}
If $g \geq 2$, then $\mathsf{TC}^{\mathbb{Z}_2}_{\mathsf{effv},n}(\Sigma_g^\mathsf{ref}) \in \{ 2n,2n+1\}$; and $\mathsf{TC}^{\mathbb{Z}_2}_{\mathsf{effv},n}(\Sigma_1^\mathsf{ref})=\mathsf{TC}_n(\Sigma_1)=2n-1$.
\end{proposition}
\begin{proof}
We know that $\mathsf{TC}^{\mathbb{Z}_2}_{\mathsf{effv},n}(\Sigma_g^\mathsf{ref}) \leq \mathsf{TC}_n(\Sigma_g)$. On the other hand, the action of $\mathbb{Z}_2$ on the integral
cohomology of $\Sigma_g$ is such that $a_i\mapsto -a_i$ and $b_i\mapsto b_i$. Therefore in mod-2 cohomology this action becomes trivial, and
the zero divisors of the usual diagonal are effective zero divisors.

When $g=1$, the homological lower bound of $\mathsf{TC}^{\mathbb{Z}_2}_{\mathsf{effv},n}(\Sigma_1^\mathsf{ref})$ yields the desired result as
the value of $\mathsf{TC}_n(\Sigma_1)$ is realized by the same homological lower bound.

When $g>1$, we take the usual zero divisor except one. More precisely, we use
$A_i = a_{1,1} - a_{i,i}, B_i = b_{1,1} - b_{i,i}$ for $i\geq 2$, and $C=a_{2,1} - a_{2,2}$. They are effective zero divisors
and $C A_2 B_2\cdots A_n B_n \neq 0$. 
\end{proof}

Note that in this case the exact value of $\mathsf{TC}^{\mathbb{Z}_2}_{\mathsf{effv},n}(\Sigma_g^\mathsf{ref})$ may be decided by means of obstruction theory. However,
the corresponding calculations are complicated and will appear elsewhere.

%%%%%%%%%%%%%%%%%%%%%%%%%%%%%%%%%%%%%%%%%%%%%%%%%%%%%%%%%%%%%%%%%%%%%%%%%%%%%
%%%%%%%%%%%%%%%%%%%%%%%%%%%%%%%%%%%%%%%%%%%%%%%%%%%%%%%%%%%%%%%%%%%%%%%%%%%%%

\subsection{The rotation action}
Recall that the rotation action on $\Sigma_{2l+1}$ has as quotient $\Sigma_{l+1}$. 

\begin{theorem}
Let $l\geq 0$. For the rotation action on $\Sigma_{2l+1}$, the three terms in Proposition~\ref{inequality} are equal, that is,
\[ \mathsf{TC}^{\mathbb{Z}_2}_{\mathsf{effv},n}(\Sigma_{2l+1}^{\mathsf{rot}}) = \mathsf{TC}^{\mathbb{Z}_2}_{\mathsf{effl},n}(\Sigma_{2l+1}^{\mathsf{rot}}) = \mathsf{TC}_n(\Sigma_{l+1}). \]
\end{theorem}
\begin{proof} 
The action of $\mathbb{Z}_2$ on the cohomology of $\Sigma_{2l+1}^{\mathsf{rot}}$ interchanges $a_1$ and $a_{2l+1}$ and likewise $b_1$ and $b_{21+1}$,
and leaves $a_{l+1}$ and $b_{l+1}$ fixed.

When $l=0$, the restriction of $\Delta^{\mathbb{Z}_2}_n$ to each component of $M^{\mathbb{Z}_2}_n(\Sigma_1^{\mathsf{rot}})$ is
the usual diagonal. Thus the homological lower bound of $\mathsf{TC}^{\mathbb{Z}_2}_{\mathsf{effv},n}(\Sigma_{1}^{\mathsf{rot}})$ is the same as that of $\mathsf{TC}_n(\Sigma_1)$.

When $l>0$, let
\begin{align*}
A_i &=(a_{1,1} + a_{2l+1,1})- (a_{1,i}  + a_{2l+1,i} ) = (a_1+a_{2l+1})\otimes 1\otimes \cdots\otimes 1 - 1\otimes \cdots \otimes(a_1 + a_{2l+1} )\otimes \cdots\otimes 1 \\
B_i &= (b_1+b_{2l+1})\otimes 1\otimes \cdots\otimes 1 - 1\otimes \cdots \otimes(b_1 + b_{2l+1} )\otimes \cdots\otimes 1 \\
C_1 &= a_{l+1}\otimes 1\otimes \cdots\otimes 1 - 1\otimes a_{l+1} \otimes 1 \cdots \otimes 1 \\
C_2 &= b_{l+1}\otimes 1\otimes \cdots\otimes 1 - 1\otimes b_{l+1} \otimes 1 \cdots \otimes 1 \\   
\end{align*}
where $i\geq 2$ and both $a_1 + a_{2l+1}$ and $b_1 + b_{2l+1}$ are in the $i^{th}$ factor.
These are nontrivial classes in the kernel of $\Sigma^{\mathbb{Z}_2}_n$ and 
satisfy $C_1 C_2 A_2 B_2 \cdots A_n B_n = 2^{n} c\otimes \cdots\otimes c$.
The result follows since the value of $\mathsf{TC}_n(\Sigma_l)$ is the same as the value afforded by the homological lower bounds that we just obtained.
\end{proof}

\begin{proposition}
$\mathsf{TC}^{\mathbb{Z}_2}_{\mathsf{orb},n}(\Sigma_{2l+1}^{\mathsf{rot}}) = \mathsf{TC}_n(\Sigma_{l+1})$ when $l>0$, and 
$\mathsf{TC}^{\mathbb{Z}_2}_{\mathsf{orb},n}(\Sigma_1) \in \{ 2n, 2n+1  \}$.
\end{proposition}
\begin{proof}
When $l>0$ we have $\mathsf{TC}^{\mathbb{Z}_2}_{\mathsf{orb},n}(\Sigma_{2l+1}^{\mathsf{rot}}) = 2n+1 = \mathsf{TC}_n(\Sigma_{l+1})$. When $l=0$, the following
commutative diagram (induced by the quotient map $Sigma_{2l+1} \to \Sigma_{l+1}$)
\[
\xymatrix{
P(\Sigma_1^{\mathsf{rot}}) \ar[rr] \ar[d] & & P(\Sigma_1^{\mathsf{rot}}/\mathbb{Z}_2) \ar[d] \\
\Sigma_1^{\mathsf{rot}} \times \Sigma_1^{\mathsf{rot}} \ar[rr] & & \Sigma_1^{\mathsf{rot}}/\mathbb{Z}_2 \times \Sigma_1^{\mathsf{rot}}/\mathbb{Z}_2 
}
\]
can be used to see that $b_{1,1}\ldots b_{1,n}$ and $a_{1,1} + a_{1,i}$ are mod-2 orbital zero divisors ($i\geq 2$). 
Moreover $b_{1,1}\cdots b_{1,n}(a_{1,1} + a_{1,2})\cdots (a_{1,1} + a_{1,n}) \neq 0$. The result follows.
\end{proof}

%%%%%%%%%%%%%%%%%%%%%%%%%%%%%%%%%%%%%%%%%%%%%%%%%%%%%%%%%%%%%%%%%%%%%%%%%%%%%
%%%%%%%%%%%%%%%%%%%%%%%%%%%%%%%%%%%%%%%%%%%%%%%%%%%%%%%%%%%%%%%%%%%%%%%%%%%%%

\subsection{The antipodal action}

Now we consider the antipodal action on $\Sigma_g$ with quotient $N_{g+1}$. 
The mod-2 cohomology ring of $N_{g+1}$ has generators $x_i$ in degree one with $1\leq i \leq g$ and $y$ in degree two subject to the relations $x_i x_j =\delta_{ij} y$. This information will be used to provide lower bounds for the sectional category of the map that defines the effectual 
topological complexity of $\Sigma_g^{\mathsf{ant}}$. The action of $\mathbb{Z}_2$ on the cohomology of $\Sigma_g^{\mathsf{ant}}$
sends $a_1$ to  $-a_g$ and interchanges $b_1$ and $b_g$.

The effective TC when $n=2$ was considered in \cite{cg}. Their calculations can easily be extended to cases when $n>2$. 

\begin{proposition}
$\mathsf{TC}^{\mathbb{Z}_2}_{\mathsf{effv},n}(\Sigma_1^{\mathsf{ant}}) = 2n -1 $, and $\mathsf{TC}^{\mathbb{Z}_2}_{\mathsf{effv},n}(\Sigma_g^{\mathsf{ant}})\in \{ 2n, 2n+1\}$ when $g\geq 2$.
\end{proposition}

Now we will focus on effectual TC of $\Sigma_g^{\mathsf{ant}}$. Let us consider the case when $g=1$. The map $\pi: \Sigma_1 \to N_{2}$ satisfies: 
$\pi^*( x_i) = a_1 + b_1$, for $i=1,2$. Recall that $\mathsf{TC}^{\mathbb{Z}_2}_{\text{effl},n}(\Sigma_g)$ is bounded below by the cup-length of the kernel of the map induced in cohomology by $1\times \pi^{n-1} \colon \Sigma_1  \to \Sigma_1\times (N_2)^{n-1}$.  We will use double indices to denote 
classes in the cohomology of $\Sigma_1 \times (N_2)^{n-1}$ of the form $1\otimes\cdots\otimes 1\otimes u \otimes 1\otimes \cdots\otimes 1$ so 
that the second index stands for the location of the class $u$ in the tensor product.

We have: 
\[
\begin{aligned} 
(1\times\pi^{n-1})^*(a_{1,1} + b_{1,1} + x_{1,2})=0, \\ 
(1\times \pi^{n-1})^*(x_{1,j} + x_{1,j-1})=0, \\ (1\times\pi^{n-1})^*(x_{1,j} + x_{2,j-1})=0,  
\end{aligned}
\]
for $3 \leq j\leq n$, and
\[
\begin{aligned} 
(a_{1,1} + b_{1,1} + x_{1,2} )^3 (x_{1,3} + x_{1,2})(x_{1,3} + x_{2,2}) \cdots    (x_{1,n} + x_{1,n-1})(x_{1,n} + x_{2,n-1})\\
= (a_{1,1} + b_{1,1} )x_{1,2}^2 (x_{1,3}^2 + x_{1,3}x_{2,2} + x_{1,2}x_{1,3}) \cdots (x_{1,n}^2 + x_{1,n}x_{2,n-1} + x_{1,n-1}x_{1,n})\\
= (a_{1,1} + b_{1,1} )x_{1,2}^2 x_{1,3}^2 \cdots x_{1,n}^2 \neq 0.
\end{aligned}
\]
Therefore $2n\leq \mathsf{TC}^{\mathbb{Z}_2}_{\text{effl},k} (\Sigma_{1}^{\mathsf{ant}})$. These calculations can be carried to values of $g\geq 2$ 
to obtain the following result:

\begin{theorem}
If $g \geq 1$, then $2n\leq \mathsf{TC}^{\mathbb{Z}_2}_{\mathsf{effl},k} (\Sigma_{g}^{\mathsf{ant}}) \leq \mathsf{TC}_n(N_{g+1}) = \mathsf{TC}^{\mathbb{Z}_2}_{\mathsf{orb},n}(\Sigma_g^{\mathsf{ant}}) = 2n+1$.
\end{theorem}

\begin{theorem}
$\mathsf{TC}_{\mathsf{effl},n}^{\mathbb{Z}_2}(\Sigma_1^{\mathsf{ant}}) = 2n$.
\end{theorem}
\begin{proof}
We will now describe an explicit higher effectual motion planner $s$ that realizes $2n$ domains of continuity. Our construction is an extension of the construction in \cite{cggz}, but we make some important modifications which allow for the extension. For ease of reference, we will use similar notation when possible. Diagrams will be provided for $n = 3$.

First, consider the torus $\Sigma_1^{\mathsf{ant}} = S^1 \times S^1 \subseteq \mathbb{C} \times \mathbb{C}$. Any point $x \in \Sigma_1^{\mathsf{ant}}$ has two factors, $(x',\, x'')$, each corresponding to a copy of $S^1$. The horizontal component will be given by $x'$ and the vertical component by $x''$. Thus, we define $V_x = \{x'\} \times S^1$ and $H_x = S^1 \times \{x''\}$ as vertical and horizontal 1-dimensional sets which extend perpendicularly from $x$. See Figure \ref{fig:V-H}. In this construction, the antipodal action $\sigma$ acts s.t. $\sigma(x) = (-x', \overline{x''})$, where $\overline{x''}$ is the complex conjugate of $x''$.

Our motion planner $s$ has input of the form $a = (x,z_1,z_2,\dots,z_{n-1}) \in \Sigma_1^{\mathsf{ant}} \times (\Sigma_1^{\mathsf{ant}}/\mathbb{Z}_2)^{n-1}$ and outputs a multipath, $s(a) \in (\Sigma_1)^{J_n}$. The multipath $s(a) = (c_x,\alpha_1,\dots,\alpha_{n-1})$, where $c_x$ is the constant path at $x$, is based at the point $x \in \Sigma_1^{\mathsf{ant}}$. We will describe the recipe for each component, $\alpha_i$.

Since effectual motion planning only requires motion planning to a single representative, our path $\alpha_i$ from $x$ to $z_i$ will be restricted to half of the torus. We formalize this by defining two vertical boundary circles $C_x^I = V_{b}$ and $C_x^D = V_{-b}$ where $b = xe^{-i\pi/2}$ and $-b = xe^{i\pi/2}$. Furthermore, $M_x$ is the half torus containing $x$ and boundary sets $C_x^I$ and $C_x^D$. Explicitly, $M_x = \{xe^{i\theta} \ | \ -\frac{\pi}{2} \leq \theta \leq \frac{\pi}{2}\} \times S^1$. See Figure \ref{fig:M_x} and Figure \ref{fig:cylinder}.

\begin{figure}
    \centering
    \includegraphics[height=5cm]{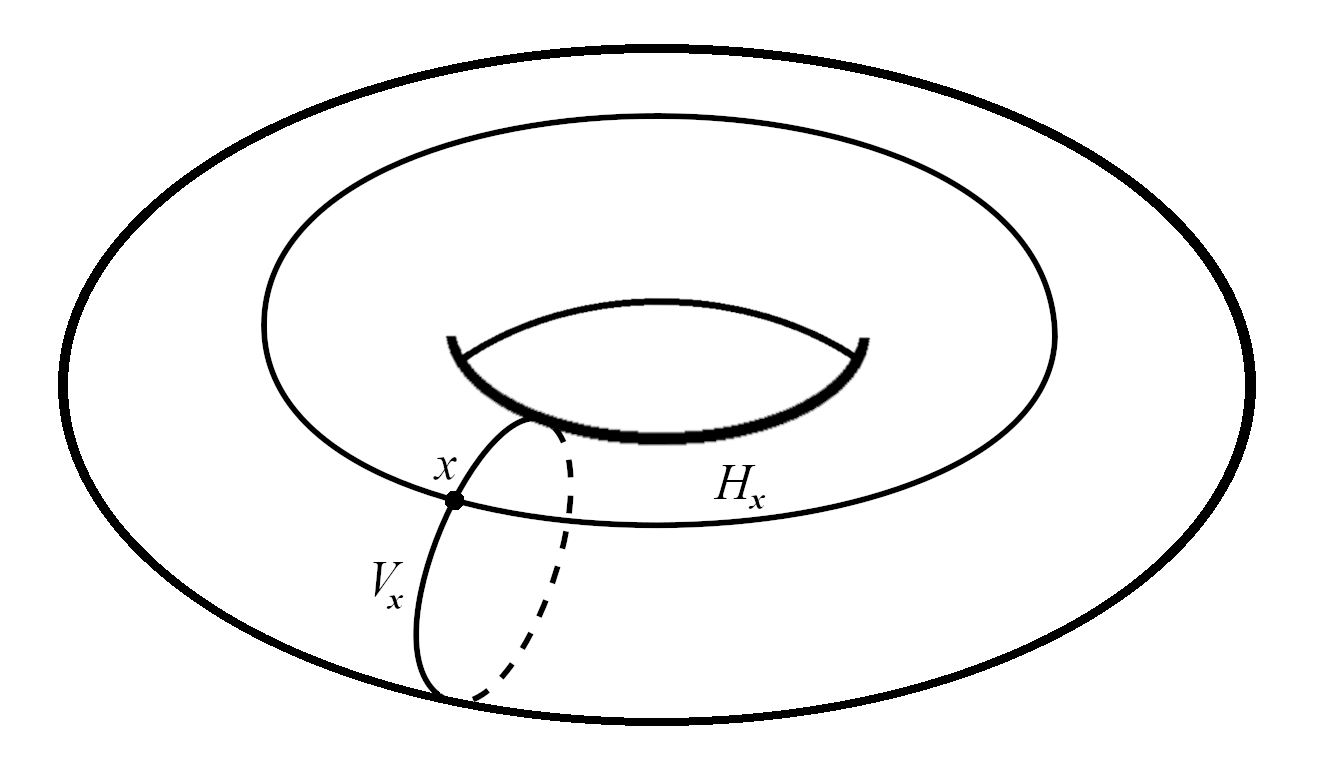}
    \caption{$V_x$ and $H_x$ for a given $x \in \Sigma_1^{\mathsf{ant}}$.}
    \label{fig:V-H}
\end{figure}

\begin{figure}
    \centering
    \includegraphics[height=5.2cm]{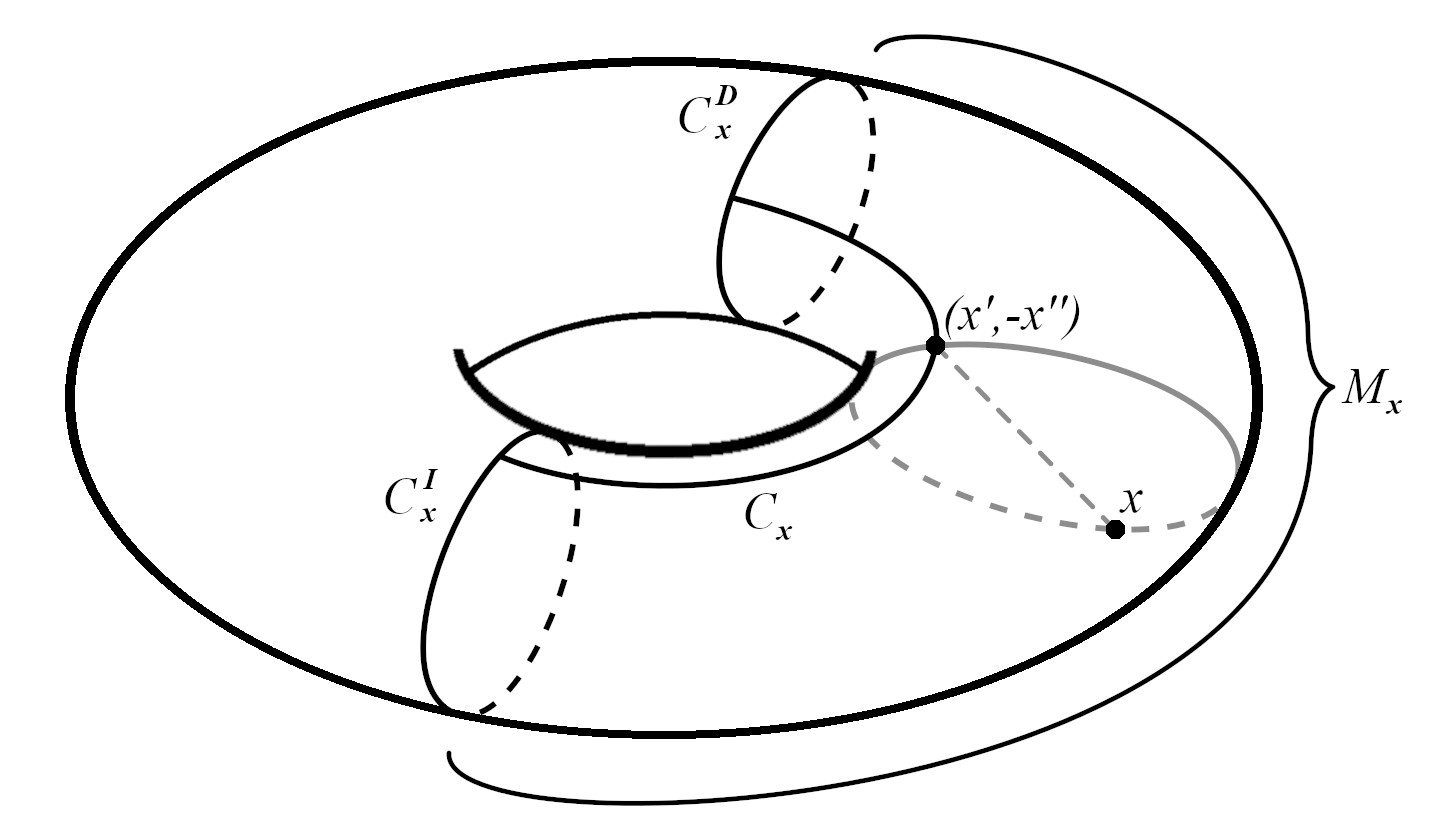}
    \caption{The 2-dimensional set $M_x$ and its important component sets for a given $x \in \Sigma_1^{\mathsf{ant}}$.}
    \label{fig:M_x}
\end{figure}

Additionally, define the following sets:
\begin{align*}
    &E^+ := S^1 \times \{1\} \\
    &E^- := S^1 \times \{-1\} \\
    &E := E^- \sqcup E^+ \text{, the equator of the torus.} \\
    &C_x := (S^1 \times \{-x''\}) \cap M_x \\
    &A_x := C_x^I \cup C_x \cup C_x^D \\
    &\{b_x\} := C_x^I \cap C_x \\
    &a_x := \sigma(b_x) 
\end{align*}
\begin{figure}
    \centering
    \includegraphics[height=7cm]{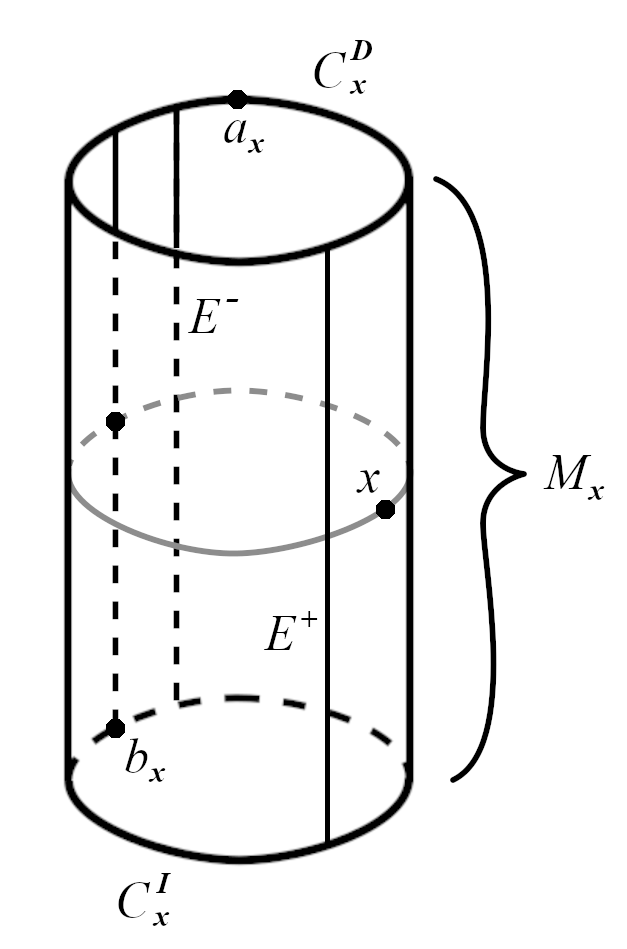}
    \caption{The cylindrical representation of $M_x$.}
    \label{fig:cylinder}
\end{figure}
The previous paragraphs have included only minor modifications of the construction given in \cite{cggz}. The challenge of higher motion is proposing an organization strategy which extends for $n \geq 2$. This challenge is overcome by what we call the characteristic numbers of $a$. For a given $a = (x,z_1,z_2,\dots,z_{n-1}) \in \Sigma_1^{\mathsf{ant}} \times (K_2)^{n-1}$, consider the canonical quotient map of the Klein bottle, $\pi\, \colon\ \Sigma_1^{\mathsf{ant}} \to K_2$. The following two functions will tell us the characteristic numbers of the elements in $a$:
\begin{align*}
    \epsilon\, \colon \ &\Sigma_1^{\mathsf{ant}} \times (K_2)^{n-1} \longrightarrow \mathbb{N} \\
    \epsilon(a) &= \begin{cases} 1 \quad \text{if } x\in E \\
    0 \quad \text{if } x \notin E 
    \end{cases}\\
    \chi_i \, \colon \ &\Sigma_1^{\mathsf{ant}} \times (K_2)^{n-1} \longrightarrow \mathbb{N} \\
    \chi_i(a) &= \begin{cases} 2 \quad \text{if } \pi^{-1}(z_i) = \{a_x,\, b_x\} \\
    1 \quad \text{if } \pi^{-1}(z_i) \cap M_x \subseteq A_x \setminus \{a_x,\, b_x\} \\
    0 \quad \text{if } \pi^{-1}(z_i) \cap M_x \subseteq M_x \setminus A_x
    \end{cases}\\
\end{align*}
Together, these functions combine to give the characteristic tuple of $a$.
\begin{align*}
    C \, \colon \ &\Sigma_1^{\mathsf{ant}} \times (K_2)^{n-1} \longrightarrow \mathbb{N}^{n} \\
    C(a) &= (\epsilon(a),\chi_1(a),\chi_2(a),\dots,\chi_{n-1}(a))
\end{align*}
The characteristic tuple $C(a)$ contains necessary information for combining domains of continuity. Furthermore, its components tell us the recipe for each $\alpha_i$. Recall, $\alpha_i$ is the path from $x$ to the class $z_i$ in the multipath $s(a)$. To ensure consistency regardless of representative, set $y_i$ as the only element in the set $\pi^{-1}(z_i) \cap (M_x \setminus C^I_x)$. Thus, $\alpha_i$ will be a path from $x$ to $y_i$. To begin, fix a clockwise orientation on all vertical circles in the torus.

When $\epsilon(a) = 0$ and $\chi_i(a) = 0$ or $\epsilon(a) = 1$ and $\chi_i(a) = 0$, $\alpha_i$ is the canonical path on the torus from $x$ to $y_i$. See Figure \ref{fig:1,1} for an example of the latter.

\begin{figure}
    \centering
    \includegraphics[height=7cm]{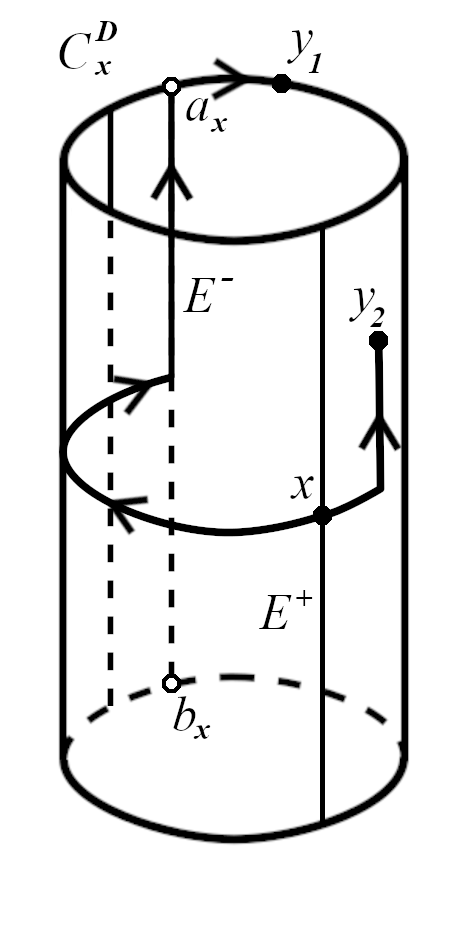}
    \caption{A multipath when $C(a) = (1,1,0)$.}
    \label{fig:1,1}
\end{figure}

When $\epsilon(a) = 0$ and $\chi_i(a) = 1$, $\alpha_i$ constitutes this path. Travel along the semicircle path from $(x',\, x'')$ to $(x',\, -x'')$ in a clockwise direction. Then, follow the arc from $(x',\, -x'')$ to $(y_i',\, -x'')$. If $y_i \notin C^D_x$, then $y_i''= -x''$ and you are finished. However, if $y_i \in C^D_x$, then follow the path from $(y_i',\, -x'')$ to $(y_i',\, y_i'')$ within $C^D_x$ s.t. the path does not contain $a_x$. See Figure \ref{fig:0,1} for two examples of $\alpha_i$ in this scenario.

\begin{figure}
    \centering
    \includegraphics[height=7cm]{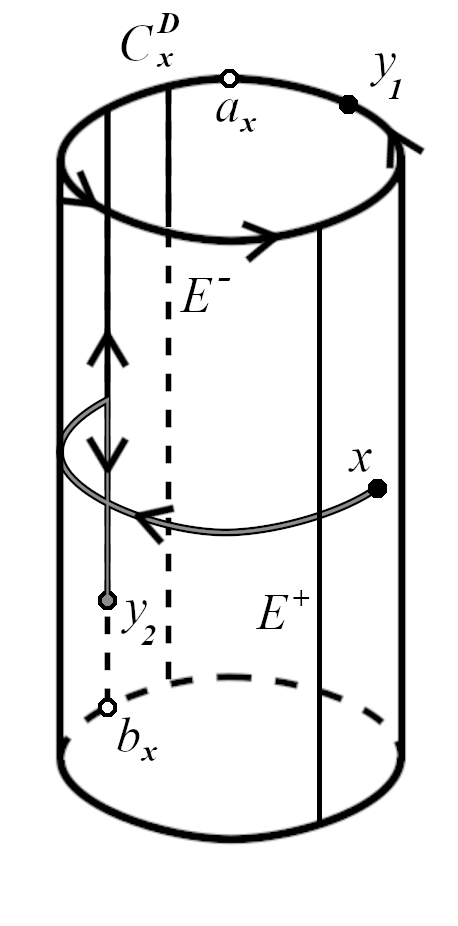}
    \caption{A multipath when $C(a) = (0,1,1)$.}
    \label{fig:0,1}
\end{figure}

When $\epsilon(a) = 0$ and $\chi_i(a) = 2$, construct $\alpha_i$ in the following way: first, travel along the semicircle path from $(x',\, x'')$ to $(x',\, -x'')$ in a clockwise direction. Then, traverse the geodesic from $(x',\, -x'')$ to $(y_i',\, -x'')$. Finally, follow a clockwise path from $(y_i',\, -x'')$ to $y_i = a_x$. See Figure \ref{fig:0,2,2}.

\begin{figure}
    \centering
    \includegraphics[height=7cm]{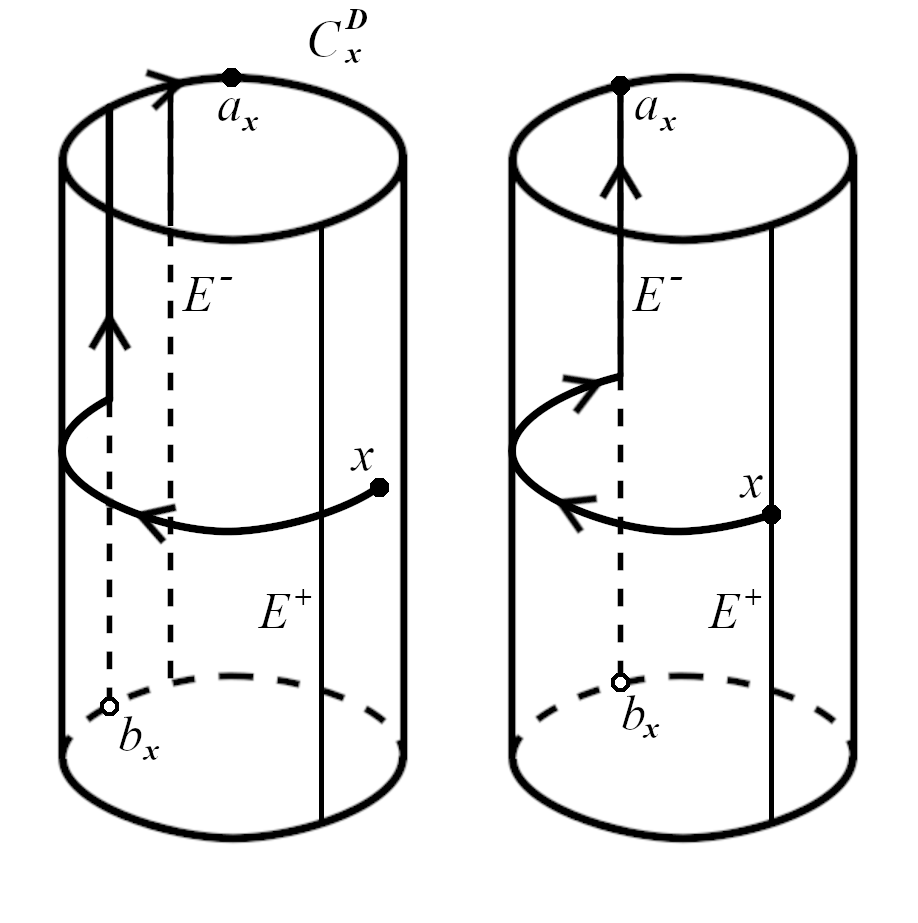}
    \caption{Multipaths when $C(a) = (0,2,2)$ and $C(a) = (1,2,2)$ respectively.}
    \label{fig:0,2,2}
\end{figure}

When $\epsilon(a) = 1$ and $\chi_i(a) = 1$, construct $\alpha_i$ by following the semicircle path from $(x',\, x'')$ to $(x',\, -x'')$ in a clockwise direction. Then, follow the arc from $(x',\, -x'')$ to $(y_i',\, -x'')$. If $y_i \notin C^D_x$, then $y_i''= -x''$ and you are finished. However, if $y_i \in C^D_x$, then follow the path clockwise from $(y_i',\, -x'')$ to $(y_i',\, y'')$ within $C^D_x$. Refer to Figure \ref{fig:1,1}.

When $\epsilon(a) = 1$ and $\chi_i(a) = 2$, $\alpha_i$ follows the semicircle path from $(x',\, x'')$ to $(x',\, -x'')$ in the clockwise direction. Then, travel on the shortest geodesic from $(x',\, -x'')$ to $y_i = a_x$. See Figure \ref{fig:0,2,2}.

As was shown in \cite{cggz}, each $\alpha_i$ recipe is continuous over any points with the same characteristic tuple. We now show that our higher effectual motion planner has $2n$ domains of continuity. We denote $\Sigma C(a)$ as the sum of the characteristic tuple $C(a)$. For example, if $C(a) = (1,1,\dots,1)$, then $\Sigma C(a) = n$. Construct a partition $\bigsqcup \{D_0, D_1, \dots, D_{2n-1}\} = \Sigma_1^{\mathsf{ant}} \times (K_2)^{n-1}$ and define $D_t = \{a \in \Sigma_1^{\mathsf{ant}} \times (K_2)^{n-1} \ | \ \Sigma C(a) = i\} $. Note that $0 \leq i \leq 2n-1$, and we may discard the requirement of open sets for topological complexity in favor of partitions since both $\Sigma_1^{\mathsf{ant}}$ and $K_2$ are compact, metric ANRs and $e$ is a fibration (see \cite{cggz} and \cite{pp}).

Suppose $\Sigma C(a) = \Sigma C(b) = t$ for $a \neq b$. Then either $C(a) = C(b)$ or $C(a) \neq C(b)$. In the first case, we can either construct a sequence $(a,a_1,a_2,\dots,a_r,\dots) \xrightarrow{}$ $b$ within $D_t$, or we cannot. If it is not possible, then $a$ and $b$ are not at risk of continuity issues. If it is possible, then construct such a sequence and note that $C(a) = C(a_j)$ for all $j$. Thus, the sequence is a series of perturbations of $a$, and we know that each component $c_x$ or $\alpha_i$ of the multipath $s(a)$ is continuous under perturbations of $a$, so long as the characteristic tuple remains constant. Thus, $s|_{D_t}$ will remain continuous in this case. Now suppose $C(a) \neq C(b)$; then, for some $j,k \leq n-1$, one of the following will hold:
\begin{align*}
    (1)&\ \epsilon(a) > \epsilon(b) \text{ and } \chi_j(a) < \chi_j(b) \\
    \text{or } (2)&\  \epsilon(a) < \epsilon(b) \text{ and } \chi_j(a) > \chi_j(b) \\
    \text{or } (3)&\ \chi_k(a) > \chi_k(b) \text{ and } \chi_j(a) < \chi_j(b) \\
    \text{or } (4)&\ \chi_k(a) < \chi_k(b) \text{ and } \chi_j(a) > \chi_j(b)
\end{align*} 
We will elaborate on $(1)$, and the rest follow by a similar argument. Since $\epsilon(a) > \epsilon(b)$, then the first element of $a$ belongs to the equator set $E$ and the first element of $b$ does not. We can conclude that a sequence starting from $a$ in $D_t$ cannot approach $b$, as any sequence will have the first element approaching the equator set. Similarly, since $\chi_j(a) < \chi_j(b)$, a sequence starting from the $b$ cannot approach $a$ because of the dimensional restriction to the $j+1$-th element of $b$. Thus, we conclude that disjoint subsets of $D_t$ have no limit point crossovers so $s|_{D_t}$ is continuous. Therefore, we have constructed a higher, effectual motion planner for the Torus with $2n$ domains of continuity. We conclude the proof as the motion planner demonstrates $TC_{\text{effl},n}^{\mathbb{Z}_2}(\Sigma_1^{\mathsf{ant}}) \leq 2n$.
\end{proof}

%%%%%%%%%%%%%%%%%%%%%%%%%%%%%%%%%%%%%%%%%%%%%%%%%%%%%%%%%%%%%%%%%%%%%%%%%%%%%
%%%%%%%%%%%%%%%%%%%%%%%%%%%%%%%%%%%%%%%%%%%%%%%%%%%%%%%%%%%%%%%%%%%%%%%%%%%%%

\section{Appendix: Motion Planning in Euclidean Spaces}

The motion planning problem in $\mathbb{R}^3$ in the presence of obstacles was first addressed in \cite{farber04}.
If $Q_r$ stands for a nonempty set of $r$ points in $\mathbb{R}^3$, then $\mathbb{R}^3 - Q_r$ is homotopy equivalent to a bouquet of 
2-dimensional spheres and hence $TC_2(\mathbb{R}^3 \setminus Q_r) = 3$. However, the motion planner described in \cite[Example 10.4]{farber04} 
is unstable in one of its domains. Here we revisit this planner and show how to fix it.

For brevity, we will describe only the motion planner in the domain where it is not stable.  Let $Q_r = \{ p_1,\ldots, p_r\}$ and let 
$F_2 \subseteq (\mathbb{R}^3 \setminus Q_r)^2$ be the set of all pairs $(a, b)$ such that the straight line segment $[a, b]$ intersects $Q_r$ but 
this segment is not parallel to the $z$-axis. Pick $\epsilon >0$ such that  $|x -y| > \epsilon$ for any two points $x,y$ in $Q_r$. 
The motion planner on $F_2$ is given as follows: given $(a,b)$ in $F_2$ go from $a$ along the straight line segment $[a, b]$ until the distance 
to one of the obstacles $p_{i_r}$ becomes $\epsilon/2$, then move along the upper semicircle of radius $\epsilon/2$ with center at 
$p_{i_r}$ lying in the plane that contains the points $a, b$ and is parallel to the $z$-axis; then continue traveling towards $b$
in this manner.

The issue is seen when we take a sequence of elements in $F_2$, such that the sequence begins with only one obstruction and ends with two. 
An example is shown in Figure \ref{fig:Farber}.\\

\begin{figure}[hbt!]
    \centering
    \includegraphics[height=8cm]{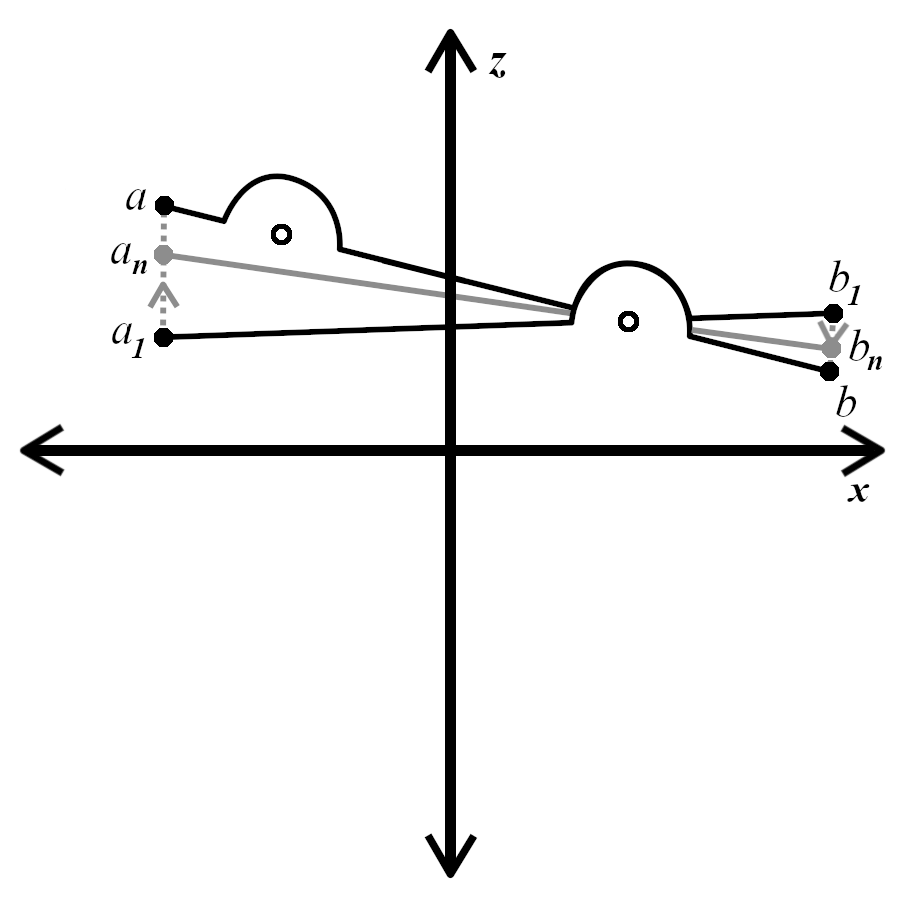}
    \caption{A perturbation of elements in $F_2$ as defined by Farber in \cite{farber04}.}
    \label{fig:Farber}
\end{figure}

Clearly, the sequence $(a_n,b_n)$ and its limit point $(a,b)$ live entirely within $F_2$. However, the paths generated by the motion planner 
do not converge to the path assigned to the limit point. Therefore, this planner is not continuous on $F_2$.

We will now describe an explicit motion planner for $\mathbb{R}^3 \setminus Q_r$ which we later extend to $\mathbb{R}^m \setminus Q_r$. 
Our motion planner will always send the robot to the origin before proceeding to the destination. 

We will assume, without loss of generality, that the set of obstacles $Q_r$ is positioned in $\mathbb{R}^3$ in such a way that $Q_r$ is not contained in a line that passes through the origin and for all $p \in Q_r$, $\pi_{xy}(p) \neq (0,0)$, where $\pi_{xy}$ is the projection onto the $xy$-plane. 

\begin{definition}
Given any obstacle point $p \in Q_r$, the radial $L_p$ is the line through the origin in $\mathbb{R}^m$ that contains the obstacle point $p$.
\end{definition}

Define $\mathcal{L}$ as the set of all radials for the obstacles in $Q_r$. Note that the set $\pi_{xy}(L)$ consists of a set of lines that go 
through the origin. Let $\theta$ be half of the smallest angle determined by these lines (see Figure~\ref{fig:angles}).

\begin{figure}
    \centering
    \includegraphics[height=9cm]{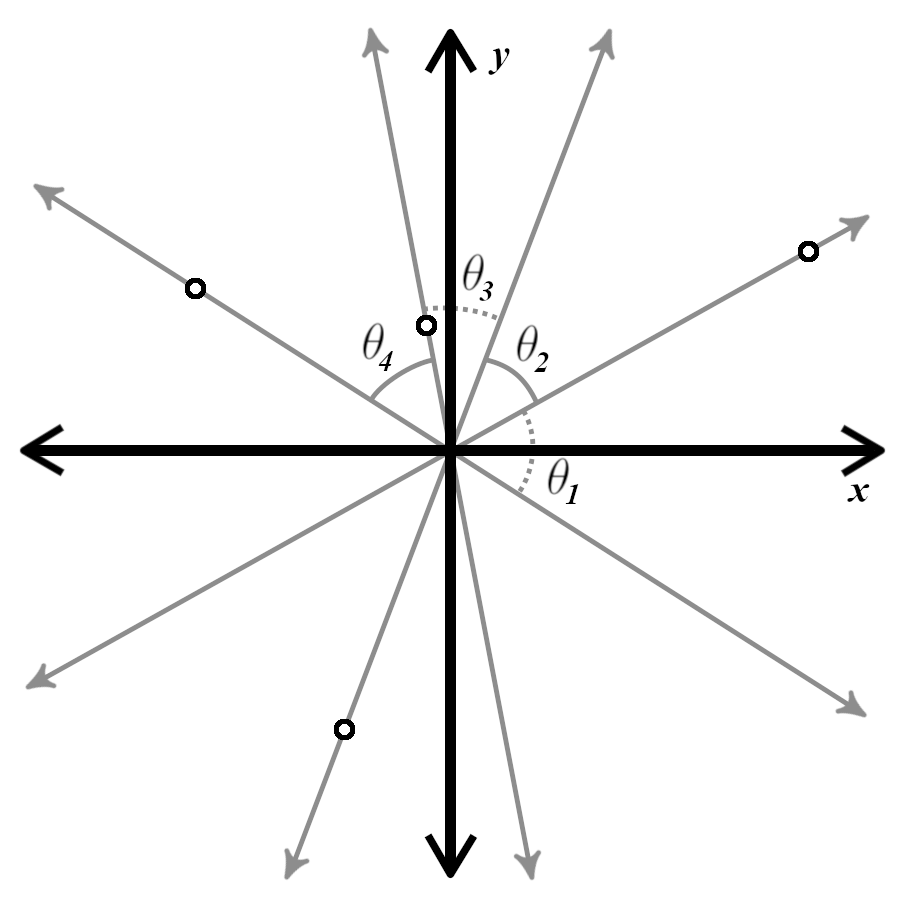}
    \caption{The angle between each projected radial in the $xy$-plane.}
    \label{fig:angles}
\end{figure}

Let $A_\theta$ be the orthogonal matrix that rotates $\mathbb{R}^3$ around the $z$-axis through an angle of $\theta$.  
Define a partition, $\{F_1,\, F_2,\,F_3\}$, on the configuration space $(\mathbb{R}^3\setminus Q_r) \times (\mathbb{R}^3\setminus Q_r)$ as follows: 
    
\begin{enumerate}
\item $F_1$ is the set of pairs $(a,b)$ such that neither $a$ nor $b$ lie on a radial.
\item $F_2$ is the set of pairs $(a,b)$ such that either $a$ lies on a radial or $b$ lies on a radial, but not both.
\item $F_3$ is the set of pairs $(a,b)$ such that both $a$ and $b$ lie on radials. 
\end{enumerate}

Now, we describe the motion planner. Suppose $(a,b) \in F_1$, then let $s_1(a,b)$ be the path from $a$ on the straight line segment to the 
origin and from the origin on the straight line segment to $b$. Suppose $(a,b) \in F_2$, then we have two possibilities. Either $a$ is on a 
radial and $b$ is not, or $b$ is on a radial and $a$ is not. In the first case, $s_2(a,b)$ defines the following path:

\[
s_2(a,b)(t) = \begin{cases}
    A_{4t\theta}a \quad &0 \leq t \leq \frac{1}{4}\\
    (2-4t)A_{\theta}a \quad &\frac{1}{4} \leq t \leq \frac{1}{2}\\
    (2t-1)b \quad &\frac{1}{2} \leq t \leq 1
\end{cases}
\]
This path rotates the point around the $z$-axis before proceeding to the origin. This ensures that the point escapes from the radial. In the second case, where $b$ lies on a radial, $s_2(a,b)(t)$ defines the following path:
\[
s_2(a,b)(t) = \begin{cases}
    (1-2t)a \quad &0 \leq t \leq \frac{1}{2} \\
    (4t-2)A_{\theta}b \quad &\frac{1}{2} \leq t \leq \frac{3}{4}\\
    A_{4(1-t)\theta}b \quad &\frac{3}{4} \leq t \leq 1\\
\end{cases}
\]
\begin{figure}
    \centering
    \includegraphics[height=8cm]{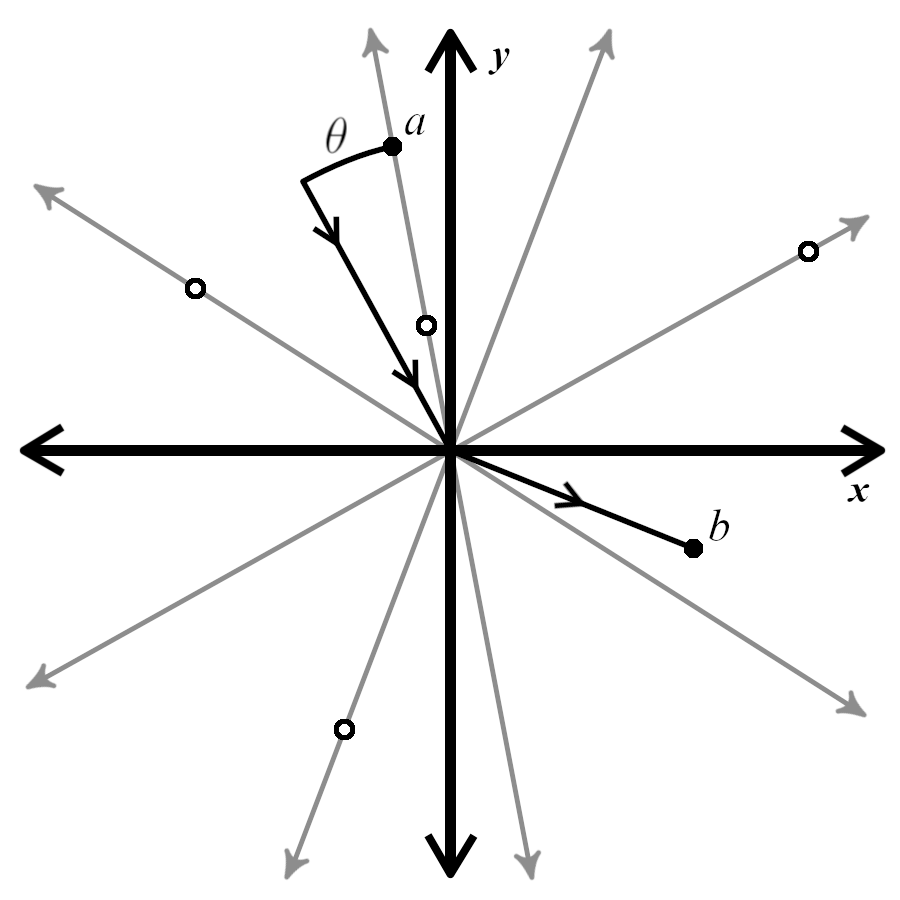}
    \caption{Motion planning for $s_2(a,b)$ projected onto the $xy$-plane.}
    \label{fig:xy-motion}
\end{figure}
See Figure \ref{fig:xy-motion} for an example. Suppose $(a,b) \in F_3$. Then, $s_3(a,b)(t)$ defines the following path:
\[
s_3(a,b)(t) = \begin{cases}
    A_{4t\theta}a \quad &0 \leq t \leq \frac{1}{4}\\
    (2-4t)A_{\theta}a \quad &\frac{1}{4} \leq t \leq \frac{1}{2}\\
    (4t-2)A_{\theta}b \quad &\frac{1}{2} \leq t \leq \frac{3}{4}\\
    A_{4(1-t)\theta}b \quad &\frac{3}{4} \leq t \leq 1\\
\end{cases}
\]

Note that every $s_j$ is well-defined. It is easy to see that $s_1$ and $s_3$ are continuous over their domains, and $s_2$ is continuous over each individual scenario above. However, we must show that $s_2$ is continuous over both subsets of $F_2$. Define $A$ as the set of points $(a,b) \in (\mathbb{R}^3\setminus Q_r)^2$ such that $a$ lies on a radial and $b$ does not, and $B = F_2 \setminus A$.

Consider any sequence $(a_j,b_j)$ in $A$. Note that $a_k$ for all $k$ must lie on a radial. Thus, the sequence must converge to an element $(a,b)$ with $a$ on a radial. Thus, the limit points of $A$ must be in $A$ or $F_3$. Therefore, $\partial A \subseteq F_3$.
Similarly, the limit points of $B$ must be in $B$ or $F_3$. Therefore, $\partial B \subseteq F_3$. Finally, we conclude that $s_2$ is a 
continuous motion planner since $F_2 = A \sqcup B$ and there are no limit point cross-overs between $A$ and $B$.

Thus, we have shown that our motion planner $s$ for $\mathbb{R}^3\setminus Q_r$ is stable when restricted to each $F_j$. Also note that $s$ is a symmetrized motion planner, as defined in \cite{bgrt}. In fact, it is possible to extend the above construction to $\mathbb{R}^m$ and also
to higher symmetrized TC. We proceed as follows: set $X = \mathbb{R}^m \setminus Q_r$ and take any $x \in X^n$. We denote the higher, 
symmetrized motion planner $\mathbf{S}\colon X^n \to M_n(X)$.  $\mathbf{S}(x)$ is a multipath made of paths that send each 
component $x_i$ of $x$ to the origin, according to the same instructions as the motion planner $s$ when $n=2$.

Now we let $F_i$ be the set of all  $x \in X^n$  such that $x$ has $i$ components on a radial line. This defines a partition
$F_0,F_1,\ldots,F_n$ of $X^n$. Furthermore, the continuity of $\mathbf{S}$ on each $F_i$ follows immediately 
from the continuity of $s$. We leave the details to the interested reader. This planner realizes the value 
$\mathsf{TC}^\Sigma_n(\mathbb{R}^m \setminus Q_r) = n+1$.

%%%%%%%%%%%%%%%%%%%%%%%%%%%%%%%%%%%%%%%%%%%%%%%%%%%%%%%%%%%%%%%%%%%%%%%%%%%%%
%%%%%%%%%%%%%%%%%%%%%%%%%%%%%%%%%%%%%%%%%%%%%%%%%%%%%%%%%%%%%%%%%%%%%%%%%%%%%

\end{document}